\documentclass[a4paper,reqno,12pt]{amsart}

\usepackage{amssymb}
\usepackage{amsthm}
\usepackage{graphicx}
\usepackage{float}

\usepackage{tikz}
\usepackage{tkz-euclide}

\newtheorem{theorem}{Theorem}[section]
\newtheorem{lemma}[theorem]{Lemma}
\newtheorem{proposition}[theorem]{Proposition}
\newtheorem{corollary}[theorem]{Corollary}

\newtheorem{definition}[theorem]{Definition}

\newtheorem{example}[theorem]{Example}

\newtheorem*{citelemma}{Lemma}

\newtheorem*{remark}{Remark}
\numberwithin{equation}{section}

\begin{document}


\title[A new proximity function estimate]{A new proximity function estimate on the quotient of the difference and the derivative of a meromorphic function}
\author{Lasse Asikainen}
\author{Juha-Matti Huusko}
\author{Risto Korhonen}
\address{Department of Physics and Mathematics, University of Eastern Finland, P.O. Box 111, FI-80101 Joensuu, Finland}

\email{lasseasi@uef.fi}
\email{juha-matti.huusko@uef.fi}
\email{risto.korhonen@uef.fi}

\maketitle
\makeatletter

\begin{abstract}
It is shown that, under certain assumptions on the growth and value distribution of a meromorphic function $f(z)$,
\begin{equation*}
    m\left(r,\frac{\Delta_cf - ac}{f' - a}\right)=S(r,f'),
\end{equation*}
where $\Delta_c f=f(z+c)-f(z)$ and $a,c\in\mathbb{C}$. This estimate implies a lower bound for the Nevanlinna ramification term in terms of the difference operator with an arbitrary shift. As a consequence it follows, for instance, that if $f$ is an entire function of hyper-order $<1$ whose derivative does not attain a value $a\in\mathbb{C}$ often
$$N\left(r,\frac{1}{f'-a}\right)=S(r,f),$$
then the finite difference $\Delta_c f$ cannot attain the value $ac$ significantly more often
$$N\left(r,\frac{1}{\Delta_c f-ac}\right)=S(r,f).$$
Additional applications of the estimate above include a new type of a second main theorem, deficiency relations between $\Delta_cf$ and $f'$ and new Clunie and Mohon'ko type lemmas.
\end{abstract}

\section{Introduction}

The difference operator measures the change of a function between adjacent points and is often used to depict discrete or stepwise variations. In contrast, the derivative measures the rate of change of a function with respect to its independent variable, illustrating how a function changes continuously over a given interval. Although these operators are fundamentally different, formally, there are many similarities between the properties of the difference operator and the derivative. These include properties related to Wronskian and Casoratian determinants, Laurent and factorial series, solution methods for linear equations, and the Wiman-Valiron theory for differential \cite{hayman:74} and difference \cite{ishizakiy:04} operators. A systematic treatment of properties of the difference operator has been given, for instance, by Boole \cite{boole:09} and by Milne-Thomson \cite{milne-thomson:33}.

Links between the value distribution properties of the usual derivative $f'$ and those of the finite difference $\Delta_c f(z) = f(z+c)-f(z)$ for a meromorphic function $f$ are a topic of interest, not least due to the fact, that $c^{-1}\Delta_c f\to f'$ as $c\to0$ pointwise. In \cite{bergweilerl:07MR} Bergweiler and Langley establish that for fixed $c\in\mathbb{C}\setminus\{0\}$,  $c^{-1}\Delta_c f(z)\to f'(z)$ as $|z|\to\infty$ outside an $\varepsilon$-set if $f$ is meromorphic with order less than $1$, which is a very strong relation between the two operators. One is then led to seek weaker relations between $f'$ and $\Delta_c f$ that hold for a more general set of meromorphic functions. Much work has already been done to investigate the difference quotient $f(z+c)/f(z)$. For example Chiang and Feng \cite{chiangf:09} have derived asymptotic relationships among difference quotients and logarithmic derivatives for meromorphic functions of finite order.

This paper establishes the following upper bound under certain assumptions on the hyper-order of a meromorphic and non-$c$-periodic $f$:
\begin{equation}
    \label{eq:logarithmic_derivatedifference}
    m\left(r,\frac{\Delta_cf - ac}{f' - a}\right)=S(r,f),
\end{equation}
where $a\in\mathbb{C}$ is arbitrary, and where an exceptional set of finite logarithmic measure is implicitly assumed in $S(r,f)$ here and in the sequel. This result relates the value distributions of $f'$ and $\Delta_c f$ in certain ways. For example, a simple consequence of the result is that for entire $f$, it implies
$$N(r,1/\Delta_c f)\leq N(r,1/f')+S(r,f),$$
meaning that $\Delta_c f$ cannot have significantly more zeros than $f'$, for any $c$.

It is readily seen, that the lemma on the difference quotient
$$m\left(r,\frac{\Delta_cf}{f}\right)=S(r,f),$$
which is the difference analogue of the lemma on the logarithmic derivative proven independently by Halburd and Korhonen in \cite{halburdk:06JMAA} and by Chiang and Feng in \cite{chiangf:08}, follows from \eqref{eq:logarithmic_derivatedifference} and the lemma on the logarithmic derivative. However, this derivation of the lemma on the difference quotient requires slightly more restrictive hypothesis than those derived by Halburd, Korhonen and Tohge in \cite{halburdkt:14TAMS}.

If $f$ is an entire function of sufficiently regular growth and of finite order, then further consequences of \eqref{eq:logarithmic_derivatedifference} include the following relation between the deficiencies of $f'$ and $\Delta_c f$ (see Proposition \ref{prop:defect_inequality} below)
$$\delta(a,f')\leq\delta(ac,\Delta_c f), \quad a\in\mathbb{C},$$
and the following relation (see Proposition \ref{prop:c_separated_index_inequality} below) between the index of $c$-separated $a$-values $\pi_c(a,f)$ for $f$ (see Definition \ref{def:pair_counting_function} below) and the deficiency of zeros of $f'$
$$\sum_{a\in\mathbb{C}}\pi_c(a,f)\leq 1-\delta(0,f').$$

The remainder of the paper is organized as follows. Section 2 contains the main result and its immediate consequences. Section 3 contains the proof of the main result. Section 4 details some more consequences of the main result, divided into four subsections with different themes.

\section{The main result}

In this section we give the main result of the paper and its immediate consequences. In order to state the main result, we need to define the hyper-order of a meromorphic function.

\begin{definition}[Hyper-order]
If $T:\mathbb{R}_+\to\mathbb{R}_+$ is an increasing function, then the hyper-order $\xi$ of $T$ is defined by
$$\xi=\limsup_{r\to\infty}\frac{\log\log T(r)}{\log r}.$$

If $f$ is a meromorphic function, then its hyper-order is defined as the hyper-order of $T(\cdot,f)$. Hence, in particular, functions with a finite order of growth have hyper-order $0$.
\end{definition}

\begin{theorem}
\label{thm:ram_vs_pair}
Let $c\in\mathbb{C}\setminus\{0\}$, $a\in\mathbb{C}$ and let $f$ be a transcendental meromorphic function of hyper-order $\xi<1$ that is not $c$-periodic. Then for $\epsilon\in(0,1)$
\begin{equation}
\label{eq:main_res}
\begin{split}
    m\left(r, \frac{\Delta_c f - ac}{f' - a}\right)
    &= O\left(\frac{T(r,f')}{r^{1-\xi-\epsilon}}\right)\\
    &\quad+O\left(R_{\epsilon,c}\left(r,\frac{1}{f'-a}\right)\frac{N\left(r,\frac{1}{f'-a}\right)}{r^{3/2-2\xi-\epsilon}}+R_{\epsilon,c}(r,f')\frac{N(r,f')}{r^{3/2-2\xi-\epsilon}}\right)
\end{split}
\end{equation}
as $r\to\infty$ outside an exceptional set $E=E(\epsilon,a,c,f)$ of finite logarithmic measure, where
$$R_{\epsilon,c}(r,g)=\frac{n_{\angle\epsilon,c}(r,g)}{n(r,g)}$$
with $n_{\angle\epsilon}(r, g)$ counting a pole $|z_0|<r$ of $g$ according to its multiplicity only if 
$$\left|\sin\left(\text{arg}\ \frac{z_0}{c}\right)\right|\geq1-\sqrt{\epsilon}.$$
\end{theorem}

\begin{remark}
The proof of Theorem \ref{thm:ram_vs_pair} is presented in the next chapter. It utilizes many of the techniques found in other works that involve estimating difference quotients, such as \cite{chiangf:09}, \cite{halburdk:06JMAA} and \cite{bergweilerl:07}. The ideas behind using the hyper-order of $f$ originate from the Proof of Theorem 5.1 in \cite{halburdkt:14TAMS}. 
\end{remark}

\begin{remark}
In the above theorem, we see that when the hyper-order $\xi$ of $f$ satisfies $\xi<3/4$ or when
$$\max\left\{R_{\epsilon,c}(r,f'), R_{\epsilon,c}\left(r,\frac{1}{f'-a}\right)\right\}=O\left(\dfrac{1}{r^{2\xi-3/2+2\epsilon}}\right)$$ as $r\to\infty$, then the error terms in the above results are of growth class $S(r,f')$ outside the exceptional set. 

By the definition of $R_{\epsilon,c}$ and $n_{\angle\epsilon,c}$, it is seen that only the poles and $a$-points of $f'$ inside certain sectors of $\mathbb{C}$ count toward the right-most error term in \eqref{eq:main_res}. The sectors are centred around the half-lines pointing from the origin toward $ic$ and $-ic$ with central angles of $\pi-2\sin^{-1} (1-\sqrt{\epsilon})$.
\end{remark}

The following lemma by Halburd, Korhonen and Tohge from \cite{halburdkt:14TAMS} is used in this paper multiple times since it gives a very useful estimate for differences of functions with hyper-order $<1$.

\begin{lemma}[Lemma 8.3 from \cite{halburdkt:14TAMS}]
    \label{lemma:hyperorder_shift_estimate}
    Let $T:R^+\to R^+$ be an increasing continuous function of hyper-order $\xi<1$. Then if $u>0$ is fixed, we have
    $$T(r+u)-T(r)=o\left(\frac{T(r)}{r^\delta}\right)$$
    where $\delta\in(0,1-\xi)$ and $r$ runs to infinity outside exceptional set of finite logarithmic measure.
\end{lemma}

The corollary below shows that the sectors that count toward $R_{\epsilon,c}$ may be rotated arbitrarily in order to improve the bound.

\begin{corollary}
\label{thm:ram_vs_pair-corollary}
Let the assumptions of Theorem~\ref{thm:ram_vs_pair} be satisfied. Moreover, assume that there is an open sector $S$ of $\mathbb{C}$ such that $f'$ has no poles nor $a$-points in $S\cup(-S)$. Then for sufficiently small $\epsilon\in(0,1)$
\begin{equation}
\begin{split}
    m\left(r, \frac{\Delta_c f - ac}{f' - a}\right)
    = O\left(\frac{T(r,f')}{r^{1-\xi-\epsilon}}\right)
\end{split}
\end{equation}
as $r\to\infty$ outside an exceptional set $E$ of finite logarithmic measure.
\end{corollary}
\begin{proof}
If $ic\in S\cup(-S)$, the theorem follows as soon as $\epsilon$ is small enough; in that situation $n_{\angle\epsilon,c}$ and $R_{\epsilon,c}$ will be identically zero.

Assume that $ic\not\in S\cup(-S)$. We find $c_1,c_2\in iS\cup(-iS)$ such that $c=c_1+c_2$. The constants $c_1,c_2$ can be found as follows. Let $\tilde{c_1},\tilde{c_2}\in iS\setminus\{0\}$ such that $\tilde{c_1}\neq t\tilde{c_2}$ for all $t\in(0,\infty)$. Because $\tilde{c_1}$ and $\tilde{c_2}$ are linearly independent, $c=\alpha_1\tilde{c_1}+\alpha_2\tilde{c_2}$ for some $\alpha_1,\alpha_2\in\mathbb{R}\setminus\{0\}$. Choose $c_1=\alpha\tilde{c_1}$ and $c_2=\alpha\tilde{c_2}$. Then by noting that 

\begin{equation*}
\begin{split}
\Delta_c f(z) - ac&=f(z+c)-f(z)-ac\\
&=[f(z+c_1+c_2)-f(z+c_1)-ac_2]+[f(z+c_1)-f(z)-ac_1]\\
&=\Delta_{c_1} f(z)-ac_1+\Delta_{c_2} f(z+c_1)-ac_2,
\end{split}
\end{equation*}
we have
\begin{equation*}
m\left(r, \frac{\Delta_c f - ac}{f' - a}\right)
\leq 
m\left(r, \frac{\Delta_{c_1} f(z) - ac_1}{f'(z) - a}\right)
+m\left(r, \frac{\Delta_{c_2} f(z+c_1) - ac_2}{f'(z) - a}\right)+O(1),
\end{equation*}
where the right-most proximity function can be further estimated with the usual lemma on the difference quotient (for example Theorem 5.1 in \cite{halburdkt:14TAMS} works for functions of hyper-order $<1$):
\begin{equation*}
\begin{split}
m\left(r, \frac{\Delta_{c_2} f(z+c_1) - ac_2}{f'(z) - a}\right) &= m\left(r, \frac{(\Delta_{c_2} f(z+c_1) - ac_2)(f’(z+c_1)-a))}{(f'(z) – a) (f’(z+c_1)-a)}\right) \\
&\leq m\left(r, \frac{\Delta_{c_2} f(z+c_1) - ac_2}{f’(z+c_1)-a}\right) + m\left(r, \frac{f’(z+c_1)-a)}{f'(z) – a}\right) \\
&= m\left(r, \frac{\Delta_{c_2} f(z+c_1) - ac_2}{f’(z+c_1)-a}\right) + O\left(\frac{T(r,f')}{r^{1-\xi-\epsilon}}\right).
\end{split}
\end{equation*}
This estimate comes with an exceptional set of finite logarithmic measure, which we include in our overall exceptional set.
Thus, for $r>0$ outside some exceptional set of finite logarithmic measure, we have
\begin{equation*}
\begin{split}
m\left(r, \frac{\Delta_c f - ac}{f' - a}\right)
&\leq 
m\left(r, \frac{\Delta_{c_1} f(z) - ac_1}{f'(z) - a}\right)\\
&+m\left(r, \frac{\Delta_{c_2} f(z+c_1) - ac_2}{f’(z+c_1)-a}\right)+O\left(\frac{T(r,f')}{r^{1-\xi-\epsilon}}\right)\\
&=O\left(\frac{T(r,f')}{r^{1-\xi-\epsilon}}\right)+O\left(\frac{T(r,f'(z+c_1))}{r^{1-\xi-\epsilon}}\right)
\\&= O\left(\frac{T(r,f')}{r^{1-\xi-\epsilon}}\right),
\end{split}
\end{equation*}
where we have applied Theorem \ref{thm:ram_vs_pair} with $\epsilon>0$ small enough such that the counting sector of $n_{\angle\epsilon}$ in the statement of Theorem \ref{thm:ram_vs_pair} is contained in $S\cup(-S)$ for both $c_1$ and $c_2$. The final inequality follows by Lemma \ref{lemma:hyperorder_shift_estimate}, where we choose a smaller $\epsilon$ if necessary and include the exceptional set of finite logarithmic measure in ours.
\end{proof}
Theorem \ref{thm:ram_vs_pair} extends to higher order differences and derivatives.

\begin{corollary}
\label{corollary:ram_vs_pair_higher_order}
Let $f$ be a transcendental meromorphic function of hyper-order $\xi<3/4$, such that $\Delta_c^{n-1} f$ is not $c$-periodic for $n\in\mathbb{N}$. Then
\begin{equation}
\label{eq:corollary_eq1}
\begin{split}
m\left(r, \frac{\Delta_c^{n} f-ac^n}{f^{(n)}-a}\right)&=\sum_{k=1}^{n}O\left(\frac{T(r,\Delta_c^{n-k}f^{(k)})}{r^{1-4\xi/3-\epsilon}}\right)\\
&=\sum_{k=1}^{n}S(r,\Delta_c^{n-k}f^{(k)})=S(r,f')
\end{split}
\end{equation}
as $r\to\infty$ outside an exceptional set $E$ of finite logarithmic measure, where $\epsilon>0$ is a small enough constant. Moreover, if $\Delta_c^{n+k-1} f$ is not $c$-periodic for $k\in\mathbb{N}$, then
\begin{equation}
\label{eq:corollary_eq2}
m\left(r, \frac{\Delta_c^{n+k} f}{f^{(n)}}\right)= S(r,f').
\end{equation}
\begin{proof}
    Consider the case where $a=0$ in the first statement. Since the difference and derivative operators commute, i.e $\Delta_c(f')=(\Delta_c f)'$, by successively applying Theorem \ref{thm:ram_vs_pair} and taking the union of the consequent exceptional sets we get
    \begin{align*}
        &m\left(r, \frac{\Delta_c^n f}{f^{(n)}}\right)
        \leq m\left(r, \frac{\Delta_c^{n-1} f'}{f^{(n)}}\right)+m\left(r, \frac{\Delta_c^n f}{\Delta_c^{n-1} f'}\right)\\[6pt]
        &=m\left(r, \frac{\Delta_c^{n-1} f'}{f^{(n)}}\right)+m\left(r, \frac{\Delta_c(\Delta_c^{n-1} f)}{(\Delta_c^{n-1} f)'}\right)\\[6pt]
        &\leq m\left(r, \frac{\Delta_c^{n-1} f'}{f^{(n)}}\right)+O\left(\frac{T(r, \Delta_c^{n-1} f')}{r^{1-4\xi/3-\epsilon}}\right)\\[6pt]
        &\leq \cdots \leq m\left(r, \frac{\Delta_c f^{(n-1)}}{f^{(n)}}\right)+\sum_{k=1}^{n-1}O\left(\frac{T(r,\Delta_c^{n-k}f^{(k)})}{r^{1-4\xi/3-\epsilon}}\right)\\[6pt]
        &\leq \sum_{k=1}^{n}O\left(\frac{T(r,\Delta_c^{n-k}f^{(k)})}{r^{1-4\xi/3-\epsilon}}\right),
    \end{align*}
    as $r\to\infty$. This proves the first inequality in \eqref{eq:corollary_eq1}. The second inequality immediately follows. The final inequality also follows from the first by Lemma \ref{lemma:hyperorder_shift_estimate}. Applying the $a=0$ case with the function $g(z)=f(z)-az^n/n!$ yields the general case for the first statement, where $a\in\mathbb{C}$. Now \eqref{eq:corollary_eq2} follows using \eqref{eq:corollary_eq1} and the higher order variant of the lemma on the logarithmic derivative:
    $$m\left(r, \frac{\Delta_c^{n+k} f}{f^{(n)}}\right)
    \leq m\left(r, \frac{\Delta_c^{n+k} f}{f^{(n+k)}}\right)+m\left(r, \frac{f^{(n+k)}}{f^{(n)}}\right)= S(r,f').$$
\end{proof}
\end{corollary}

\begin{corollary}
\label{corollary:counting_function}
Let $f$ be a meromorphic function that satisfies the hypothesis of Corollary \ref{corollary:ram_vs_pair_higher_order} for $n\in\mathbb{N}$. Then
\begin{equation}
\label{eq:jensen_formula}
\begin{split}
    &N\left(r,\frac{1}{\Delta_c^n f - ac^n}\right)-N\left(r,\Delta_c^n f\right)\\
    &\leq N\left(r,\frac{1}{f^{(n)} - a}\right)-N\left(r,f^{(n)}\right)+\sum_{k=1}^{n}O\left(\frac{T(r,\Delta_c^{n-k}f^{(k)})}{r^{1-4\xi/3-\epsilon}}\right)\\
    &= N\left(r,\frac{1}{f^{(n)} - a}\right)-N\left(r,f^{(n)}\right)+S(r,f)
\end{split}
\end{equation}
as $r\to\infty$ outside an exceptional set $E$ of finite logarithmic measure.

\begin{proof}
First apply Jensen's formula to obtain
\begin{equation}
    \begin{split}
        &N\left(r,\frac{1}{\Delta_c^n f - ac^n}\right)-N\left(r,\Delta_c^n f\right)
        =\int_0^{2\pi}\log|\Delta_c^n f(re^{i\theta})-ac^n|\frac{d\theta}{2\pi}+O(1)\\[6pt]
        &=\int_0^{2\pi}\log|f^{(n)}(re^{i\theta})-a|\frac{d\theta}{2\pi}+\int_0^{2\pi}\log\left|\frac{\Delta_c^n f(re^{i\theta})-ac^n}{f^{(n)}(re^{i\theta})-a}\right|\frac{d\theta}{2\pi}+O(1)\\[6pt]
        &\leq N\left(r,\frac{1}{f^{(n)} - a}\right)-N\left(r,f^{(n)}\right)+m\left(r, \frac{\Delta_c^{n} f - ac^n}{f^{(n)} - a}\right)+O(1)
    \end{split}
\end{equation}
and then apply Corollary \ref{corollary:ram_vs_pair_higher_order}.
\end{proof}
\end{corollary}

For entire functions, the above corollary has a particularly simple special case.

\begin{corollary}
Let $f$ be an entire function of hyper-order $\xi$ that is not $c$-periodic, where $c\in\mathbb{C}\setminus\{0\}$, and let $\epsilon>0$. If $\xi < 3/4$, then
\begin{equation*}
    N\left(r,\frac{1}{\Delta_c f - ac}\right)
    \leq N\left(r,\frac{1}{f' - a}\right)+O\left(\frac{T(r,f')}{r^{1-4\xi/3-\epsilon}}\right),
\end{equation*}
as $r\to\infty$ outside an exceptional set of finite logarithmic measure. If $\xi < 1$ and 
$$N\left(r,\frac{1}{f'-a}\right)=S(r,f')$$
then also
\begin{equation*}
    N\left(r,\frac{1}{\Delta_c f - ac}\right)
    =S(r,f').
\end{equation*}
\begin{proof}
The proof of the first statement follows from case $n=1$ of Corollary \ref{corollary:counting_function} and the fact that $f$ is entire. The prove the second statement, one only needs to slightly modify to the proof of Corollary \ref{corollary:counting_function} above: since $N\left(r,\frac{1}{f'-a}\right)=S(r,f')$, the terms in Theorem \ref{thm:ram_vs_pair} that have $R_{\epsilon,c}$-terms as coefficients are equal to $S(r,f')$.
\end{proof}
\end{corollary}

\begin{remark}
The difference analogue of the lemma on the logarithmic derivative follows from Corollary \ref{corollary:ram_vs_pair_higher_order} and the lemma on the logarithmic derivative:
$$m(r,\Delta_c f/f) \leq m(r,f'/f)+m(r,\Delta_c f/f')= S(r,f)$$
outside an exceptional set of finite logarithmic measure. However this derivation of the result does not yield optimal hypothesis: for example in \cite{halburdkt:14TAMS} difference analogue of the lemma on the logarithmic derivative is proven for functions of hyper-order $<1$.
\end{remark}

\section{Proof of Theorem \ref{thm:ram_vs_pair}}

Let $S\subset[0,2\pi]$ be the subset such that
\begin{equation}
\label{eq:jensen_formula}
    \begin{split}
        \int_0^{2\pi}\log^+\left|\frac{\Delta_c f(re^{i\theta})-ac}{f'(re^{i\theta})-a}\right|\frac{d\theta}{2\pi}
        &=\int_{S}\log\left|\frac{\Delta_c f(re^{i\theta})-ac}{f'(re^{i\theta})-a}\right|\frac{d\theta}{2\pi}.
    \end{split}
\end{equation}
For all but finitely many $\theta\in S$, i.e. for such $\theta$ that the line segment $[re^{i\theta},re^{i\theta}+c]$ does not contain poles or $a$-points of $f'$, we have the following

\begin{align*}
    &\left|\frac{\Delta_c f(re^{i\theta})-ac}{f'(re^{i\theta})-a}\right|
    =\left|\frac{1}{f'(re^{i\theta})-a}\left(\int_0^c f'(re^{i\theta}+u) du-ac\right)\right|
    =\left|\int_0^c \frac{f'(re^{i\theta}+u)-a}{f'(re^{i\theta})-a} du\right|\\[6pt]
    &\leq\int_0^c\left| \frac{f'(re^{i\theta}+u)-a}{f'(re^{i\theta})-a}\right||du|
    \leq|c|\max_{t\in[0,1]}{\left|\frac{f'(re^{i\theta}+tc)-a}{f'(re^{i\theta})-a}\right|}=|c|\max_{t\in[0,1]}{\left|\frac{g'(re^{i\theta}+tc)}{g'(re^{i\theta})}\right|},
\end{align*}
where $g(z)=f(z)+az$. We substitute this in \eqref{eq:jensen_formula} and estimate the difference quotient using the Poisson-Jensen formula as is standard

\begin{equation}
\label{eq:poisson_jensen}
    \begin{split}
    &\int_0^{2\pi}\log^+\left|\frac{\Delta_c f(re^{i\theta})-ac}{f'(re^{i\theta})-a}\right|\frac{d\theta}{2\pi}-\log c\leq\int_{S}\max_{t\in[0,1]}\log\left|\frac{g'(re^{i\theta}+tc)}{g'(re^{i\theta})}\right|\frac{d\theta}{2\pi}\\[6pt]
    \leq&\int_0^{2\pi}\biggl\{\max_{t\in[0,1]}\int_0^{2\pi}\log|g'(se^{i\psi})|\text{Re}\left(\frac{2tcse^{i\psi}}{(se^{i\psi}-re^{i\theta}-tc)(se^{i\psi}-re^{i\theta})}\right)\frac{d\psi}{2\pi}\\[6pt]
    &+\max_{t\in[0,1]}\sum_{|a_k|<s}\log\left|\frac{re^{i\theta}+tc-a_k}{re^{i\theta}-a_k}\right|+\max_{t\in[0,1]}\sum_{|a_k|<s}\log\left|\frac{s^2-\overline{a_k}re^{i\theta}}{s^2-\overline{a_k}(re^{i\theta}+tc)}\right|\\[6pt]
    &+\max_{t\in[0,1]}\sum_{|b_k|<s}\log\left|\frac{re^{i\theta}-b_k}{re^{i\theta}+tc-b_k}\right|+\max_{t\in[0,1]}\sum_{|b_k|<s}\log\left|\frac{s^2-\overline{b_k}(re^{i\theta}+tc)}{s^2-\overline{b_k}re^{i\theta}}\right|\biggl\}\frac{d\theta}{2\pi}
    \end{split}
\end{equation}
where $s=\frac{\alpha+1}{2}(r+|c|)$ with $\alpha:=\alpha(r)>1$ and $\alpha(r)\to 1$ as $r\to\infty$, and where $\{a_k\}_{k\in\mathbb{N}}$ and $\{b_k\}_{k\in\mathbb{N}}$ are sequences of the zeros and poles of $g'$ respectively, ordered by modulus in ascending order and repeated according to their multiplicities.

We continue to estimate the right-hand side terms in \eqref{eq:poisson_jensen} separately. Starting enumeration from left to right and from top to bottom, we will see that for terms 1, 2 and 5 the maximization does not complicate the situation, and those terms can be handled rather conventionally. However the terms 3 and 4, where the variable of maximization $t$ is in the denominator, require some more work.

Let $\delta\in(0,1)$. Starting with term 1 on the RHS (right-hand side) of \eqref{eq:poisson_jensen}, we swap the order of integration by Fubini's theorem in order to use the well-known fact that
\begin{equation}
    \label{eq:reciprocal_integral_estimate}
    \int_0^{2\pi}\frac{1}{|re^{i\theta}-a|^\delta}\frac{d\theta}{2\pi}\leq\frac{1}{1-\delta}\frac{1}{r^\delta}
\end{equation}
for any $a\in\mathbb{C}$, which yields

\begin{equation}
    \label{eq:term1_estimate}
    \begin{split}
    &\int_0^{2\pi}\max_{t\in[0,1]}\int_0^{2\pi}\log|g'(se^{i\psi})|\text{Re}\left(\frac{2tcse^{i\psi}}{(se^{i\psi}-re^{i\theta}-tc)(se^{i\psi}-re^{i\theta})}\right)\frac{d\psi}{2\pi}\frac{d\theta}{2\pi}\\[6pt]
    \leq&\int_0^{2\pi}\int_0^{2\pi}|\log|g'(se^{i\psi})||\max_{t\in[0,1]}\left\{\frac{2t|c|s}{|se^{i\psi}-re^{i\theta}-tc||se^{i\psi}-re^{i\theta}|}\right\}\frac{d\psi}{2\pi}\frac{d\theta}{2\pi}\\[6pt]
    \leq&\frac{2|c|s}{(s-r-|c|)(s-r)^{1-\delta}}\int_0^{2\pi}|\log|g'(se^{i\psi})||\int_0^{2\pi}\frac{1}{|re^{i\theta}-se^{i\psi}|^\delta}\frac{d\theta}{2\pi}\frac{d\psi}{2\pi}\\[6pt]
    \leq&\frac{2|c|s}{(s-r-|c|)(s-r)^{1-\delta}(1-\delta)r^\delta}\int_0^{2\pi}|\log|g'(se^{i\theta})||\frac{d\theta}{2\pi}\\[6pt]
    \leq&\frac{2|c|^\delta s}{(s-r-|c|)(1-\delta)r^\delta}(m(s,g')+m(s,1/g'))\\[6pt]
    \leq&\frac{4|c|^\delta(\alpha+1)}{(1-\delta)(\alpha-1)r^\delta}(T(\alpha(r+|c|),g')+O(1))
    \end{split}
\end{equation}

For term 2 on the RHS of \eqref{eq:poisson_jensen} we may swap the order of the integral and the sum, since the sum is finite. Again we employ the estimate \eqref{eq:reciprocal_integral_estimate}, as well as using the estimate $\log(1+x)\leq x$ and the concavity of the logarithm:

\begin{equation}
    \label{eq:term2_estimate}
    \begin{split}
    &\int_0^{2\pi}\max_{t\in[0,1]}\sum_{|a_k|<s}\log\left|\frac{re^{i\theta}+tc-a_k}{re^{i\theta}-a_k}\right|\frac{d\theta}{2\pi}\\[6pt]
    \leq&\sum_{|a_k|<s}\frac{1}{\delta}\int_0^{2\pi}\max_{t\in[0,1]}\log\left(1+\left|\frac{tc}{re^{i\theta}-a_k}\right|^{\delta}\right)\frac{d\theta}{2\pi}\\[6pt]
    \leq&\sum_{|a_k|<s}\frac{1}{\delta}\log\left(\int_0^{2\pi}1+\max_{t\in[0,1]}\left|\frac{tc}{re^{i\theta}-a_k}\right|^{\delta}\frac{d\theta}{2\pi}\right)\\[6pt]
    \leq&\sum_{|a_k|<s}\frac{1}{\delta}\log\left(1+|c|^{\delta}\int_0^{2\pi}\frac{1}{|re^{i\theta}-a_k|^{\delta}}\frac{d\theta}{2\pi}\right)\\[6pt]
    \leq&\frac{|c|^{\delta}}{\delta(1-\delta)}\frac{1}{r^\delta}n(s,1/g')
     \leq\frac{2|c|^{\delta}}{\delta(1-\delta)}\frac{\alpha}{\alpha-1}\frac{1}{r^\delta}N(\alpha(r+|c|),1/g').
    \end{split}
\end{equation}
For term 5 we use the same arguments to obtain a similar estimate:

\begin{equation}
    \label{eq:term5_estimate}
    \begin{split}
    &\int_0^{2\pi}\max_{t\in[0,1]}\sum_{|b_k|<s}\log\left|\frac{s^2-\overline{b_k}(re^{i\theta}+tc)}{s^2-\overline{b_k}re^{i\theta}}\right|\frac{d\theta}{2\pi}\\[6pt]
    \leq&\sum_{|b_k|<s}\int_0^{2\pi}\log\left(1+|c|\frac{1}{|re^{i\theta}-s^2/\overline{b_k}|}\right)\frac{d\theta}{2\pi}\\[6pt]
    &\frac{2|c|^{\delta}}{\delta(1-\delta)}\frac{\alpha}{\alpha-1}\frac{1}{r^\delta}N(\alpha(r+|c|),g').
    \end{split}
\end{equation}
The 3rd and 4th RHS terms of \eqref{eq:poisson_jensen}
$$
\int_0^{2\pi}\max_{t\in[0,1]}\sum_{|a_k|<s}\log\left|\frac{s^2-\overline{a_k}re^{i\theta}}{s^2-\overline{a_k}(re^{i\theta}+tc)}\right|\frac{d\theta}{2\pi}
\quad\text{and}\quad
\int_0^{2\pi}\max_{t\in[0,1]}\sum_{|b_k|<s}\log\left|\frac{re^{i\theta}-b_k}{re^{i\theta}+tc-b_k}\right|\frac{d\theta}{2\pi},
$$ 
are both reduced to the same problem, since
$$
    \left|\frac{re^{i\theta}-b_k}{re^{i\theta}+tc-b_k}\right|
    \leq1+\left|\frac{tc}{re^{i\theta}+tc-b_k}\right|
$$
and
$$
    \left|\frac{s^2-\overline{a_k}re^{i\theta}}{s^2-\overline{a_k}(re^{i\theta}+tc)}\right|
    \leq 1+\left|\frac{tc}{re^{i\theta}+tc-s^2/\overline{a_k}}\right|.
$$
We split both into two cases. For this we fix $\epsilon\in(0,1)$. For the 4th term, we divide the sum over the poles of $g'$ into two cases:
\begin{enumerate}
    \item $P^1_\epsilon=\{b_k:\ k\in\mathbb{N}$ such that $r-|c|-\epsilon \leq |b_k| \leq r+|c|+\epsilon\}$
    \item $P^2_\epsilon=\{b_k:\ k\in\mathbb{N}$ such that $b_k\not\in P^1_\epsilon\}$
\end{enumerate}
Similarly we divide the sum in the 3rd term into two cases:
\begin{enumerate}
    \item $Z^1_\epsilon=\{a_k:\ k\in\mathbb{N}$ such that $|a_k|\geq s-\epsilon\}$
    \item $Z^2_\epsilon=\{a_k:\ k\in\mathbb{N}$ such that $a_k\not\in Z^1_\epsilon\}$
\end{enumerate}

Let us estimate the contribution to term 4 of the poles in $P^2_\epsilon$. We have the following estimate for the maximum of the sum
\begin{align*}
    &\max_{t\in[0,1]}\sum_{\substack{|b_k|<s\\b_k\in P^2_\epsilon}}\log\left|\frac{re^{i\theta}-b_k}{re^{i\theta}+tc-b_k}\right|
    \leq\frac{1}{\delta}\sum_{\substack{|b_k|<s\\b_k\in P^2_\epsilon}}\log\left(1+\max_{t\in[0,1]}\left|\frac{tc}{re^{i\theta}+tc-b_k}\right|^\delta\right)\\[6pt]
    \leq&\frac{|c|^\delta}{\delta}\sum_{\substack{|b_k|<s\\b_k\in P^2_\epsilon}}\max_{t\in[0,1]}\left|\frac{1}{re^{i\theta}+tc-b_k}\right|^\delta
    \leq\frac{|c|^\delta}{\delta}\sum_{\substack{|b_k|<s\\b_k\in P^2_\epsilon}}\max_{t\in\overline{\mathbb{D}}}\left|\frac{1}{re^{i\theta}+tc-b_k}\right|^\delta\\[6pt]
    \leq&\frac{|c|^\delta}{\delta}\sum_{\substack{|b_k|<s\\b_k\in P^2_\epsilon}}\left|\frac{1}{re^{i\theta}+|c|\frac{b_k-re^{i\theta}}{|b_k-re^{i\theta}|}-b_k}\right|^\delta,
\end{align*}
since $re^{i\theta}+tc$ lies inside the closed annulus $r-|c| \leq |z| \leq r+|c|$ for every $t\in\overline{\mathbb{D}}$ for all large enough $r$ and $b_k\in P^2_\epsilon$ lies outside that annulus, so that the optimal choice of $t$ is the unit vector in the direction of $b_k$ from $re^{i\theta}$. Further, since $b_k$ lies outside the annulus we have the estimate
\begin{align*}
    &\left|re^{i\theta}+|c|\frac{b_k-re^{i\theta}}{|b_k-re^{i\theta}|}-b_k\right| 
    = \left|1-\frac{|c|}{|b_k-re^{i\theta}|}\right|\left|re^{i\theta}-b_k\right|\\[6pt]
    &\quad\geq\left(1-\frac{|c|}{|c|+\epsilon}\right)\left|re^{i\theta}-b_k\right|,
\end{align*}
so that in total we obtain the estimate
\begin{equation}
    \label{eq:p2_estimate}
    \begin{split}
    &\int_0^{2\pi}\max_{t\in[0,1]}\sum_{\substack{|b_k|<s\\b_k\in P^2_\epsilon}}\log\left|\frac{re^{i\theta}-b_k}{re^{i\theta}+tc-b_k}\right|\frac{d\theta}{2\pi}\\[6pt]
    \leq&\frac{|c|^\delta}{\delta}\left(\frac{|c|+\epsilon}{\epsilon}\right)^{\delta}\sum_{\substack{|b_k|<s\\b_k\in P^2_\epsilon}}\int_0^{2\pi}\frac{1}{|re^{i\theta}-b_k|^\delta}\frac{d\theta}{2\pi}\\[6pt]
    \leq&\left(\frac{|c|+\epsilon}{\epsilon}\right)^{\delta}\frac{|c|^\delta}{\delta(1-\delta)}\frac{1}{r^\delta}n(s,g')\\[6pt]
    \leq&\left(\frac{|c|+\epsilon}{\epsilon}\right)^{\delta}\frac{2|c|^\delta}{\delta(1-\delta)}\frac{\alpha}{\alpha-1}\frac{1}{r^\delta}N(\alpha(r+|c|),g')
    \end{split}
\end{equation}   

Using the same reasoning to estimate the contribution to term 3 of the zeros in $Z^2_\epsilon$, we obtain the following upper bound
\begin{equation}
    \label{eq:z2_estimate}
    \begin{split}
    &\int_0^{2\pi}\max_{t\in[0,1]}\sum_{\substack{|a_k|<s\\a_k\in Z^2_\epsilon}}\log\left|\frac{s^2-\overline{a_k}re^{i\theta}}{s^2-\overline{a_k}(re^{i\theta}+tc)}\right|\frac{d\theta}{2\pi}\\[6pt]
    \leq&\frac{|c|^\delta}{\delta}\left(\frac{|c|+\epsilon}{\epsilon}\right)^{\delta}\sum_{\substack{|a_k|<s\\a_k\in Z^2_\epsilon}}\int_0^{2\pi}\frac{1}{|re^{i\theta}-s^2/\overline{a_k}|^\delta}\frac{d\theta}{2\pi}\\[6pt]
    \leq&\left(\frac{|c|+\epsilon}{\epsilon}\right)^{\delta}\frac{2|c|^\delta}{\delta(1-\delta)}\frac{\alpha}{\alpha-1}\frac{1}{r^\delta}N(\alpha(r+|c|),1/g'),
    \end{split}
\end{equation} 
because our assumption $|a_k|<s-\epsilon$ for points $a_k\in Z^2_\epsilon$ implies that $s^2/|\overline{a_k}|>r+|c|+\epsilon$, thus allowing us to use the same reasoning that we used for term 4.

This leaves us with the task of estimating the contributions of $P^1_\epsilon$ and $Z^1_\epsilon$, which turn out to be significant. For both cases we must estimate integrals of the form
$$\int_0^{2\pi}\max_{t\in[0,1]}\frac{1}{|re^{i\theta}+tc-c_k|^{\delta_k}}\frac{d\theta}{2\pi},$$
where $c_k$ is either one of the $b_k$ or $s^2/\overline{a_k}$, and where $\delta_k\in(0,1)$ can be chosen depending on $c_k$. We prove and use the following lemma to handle this integral.

\begin{lemma}

\label{lemma:p1_z1_estimate}

Let $\ 0<\epsilon,\delta<1$. Then there exists a constant $R_{\epsilon,c,\alpha}>0$ that depends on $\epsilon,\ c$ and the function $\alpha$, such that for every $r>R_{\epsilon,c,\alpha}$ 
    $$\int_0^{2\pi}\max_{t\in\mathbb{R}}\frac{1}{|re^{i\theta}+tc-c_k|^{\delta/2}}\frac{d\theta}{2\pi}
    \leq \frac{1}{1-\delta}\frac{1}{r^{\delta/2}}$$
for any $c_k\in\mathbb{C}$ and 
$$\int_0^{2\pi}\max_{t\in\mathbb{R}}\frac{1}{|re^{i\theta}+tc-c_k|^{\delta}}\frac{d\theta}{2\pi}
\leq\left(\frac{2}{\sqrt{\epsilon(2-\epsilon)}}\right)^\delta\frac{1}{1-\delta}\frac{1}{r^{\delta}}$$
for any $c_k\in P^1_\epsilon\cap D(0,s)$ and for any $c_k\in\{s^2/\overline{a_k}: a_k\in Z^1_\epsilon\cap D(0,s)\}$ with the property
$$\left|\sin\left(\text{arg}\ \frac{c_k}{c}\right)\right|<1-\sqrt{\epsilon}.$$

\begin{proof}
We seek a lower bound for the minimum over $t\in\mathbb{R}$ of the length of the line segment $[re^{i\theta}+tc, c_k]$. To this end, we orthogonally project the line segment to the 1-dimensional subspace spanned by $ic$. The geometric situation is then as presented in the Figure~\ref{fig:orthogonalProject}.
\begin{figure}[H]
    \centering{

    \begin{tikzpicture}[scale=0.85,>=stealth,
    levea/.style={line width=1.5pt}
    ]
    
    \path[clip] (-5.2,-5.2) rectangle (10,5.2);
    \begin{scope}[rotate=-37]
    \draw[levea,dashed,green!50!black] (-10,0) coordinate (ov)--(10,0);
    \draw[dashed] (-10,5)--(30,5);
    \draw[dashed,blue] (0,-10)--(0,10);
    \draw[dashed,red] (-10,-10)--(10,10);
    \draw (0,0) circle (4cm);
    
    \draw (0,0) coordinate (o);
    \draw (67:4) coordinate (A);
    \draw (5,5) coordinate (ck);
    \draw (5,0) coordinate (dk);
    \draw (0,5) coordinate (cu);
    \draw (-1,5) coordinate (cuv);
    \draw (0,6) coordinate (cuu);
    \draw (0,4.5) coordinate (c);
    
    \path[name path=line 1] (-10,0) -- (10,0);
    \path[name path=line 2] (A) --++(0,-10);
    \fill[red,name intersections={of=line 1 and line 2}] (intersection-1) coordinate (B);
    

    \tkzMarkRightAngle[size=0.4](ov,o,c);
    \tkzMarkRightAngle[size=0.4](ov,B,A);
    \tkzMarkRightAngle[size=0.4](ov,dk,ck);
    \tkzMarkRightAngle[size=0.4](cuv,cu,cuu);

    \draw[dashed] (A)--++(0,10);
    \draw[dashed] (B)--++(0,-10);
    \draw[levea,blue!50!red] (o)--(ck);
    \draw[levea,blue] (o)--(c);
    \draw[levea,red] (A)--(ck);
    
    \draw[dashed] (ck)--++(0,10);
    \draw[dashed] (dk)--++(0,-10);


    \draw[levea,red,dashed,->,shorten >=2pt] (ck)--(dk);

    \filldraw (B) circle (2pt) node[right=1mm,fill=white, fill opacity=0.8, text opacity=1]{$B=r\sin(\text{arg}(c)-\theta)\dfrac{-ic}{|c|}$};

    \draw[levea,red,dashed,->,shorten >=2pt] (A)--(B);
    \draw[levea,green!50!black] (B)--(dk);
    
    \filldraw (A) circle (2pt) node[above right=1.5mm,fill=white]{$A=re^{i\theta}$};
    \filldraw (ck) circle (2pt) node[above=2mm]{$c_k$};
    \filldraw (dk) circle (2pt) node[right=1mm]{$d_k=|c_k|\sin\left(\text{arg}\frac{c}{c_k}\right)\frac{-ic}{|c|}$};
    \filldraw (c) circle (2pt) node[right]{$c$};
    
    \end{scope}
    
    \draw (-6,0)--(13,0);
    \draw (0,-6)--(0,6);
    
    \end{tikzpicture}
    \caption{Here $A=re^{i\theta}$ and $B$ is the orthogonal projection of $A$. Similarly $d_k$ is the orthogonal projection of $c_k$.}
    }
    \label{fig:orthogonalProject}
\end{figure}
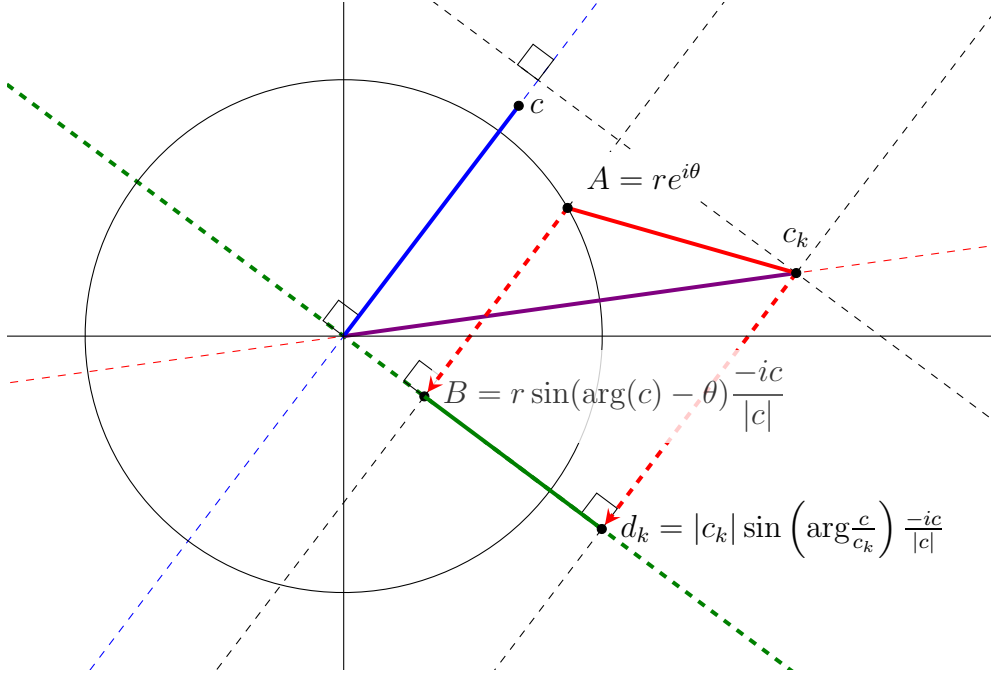

Then 
$$d_k = |c_k|\sin[\text{arg}(c)-\text{arg}(c_k)]\dfrac{-ic}{|c|}$$
is the projection of the end-point $c_k$ and 
$$B=r\sin[\text{arg}(c)-\theta]\dfrac{-ic}{|c|}$$ 
is the projection of $re^{i\theta}+tc$ for any $t\in\mathbb{R}$. Now the length of the line-segment $[B,d_k]$ is less than that of the original segment. Thus, by a simplifying change of variables, we obtain the estimate
\begin{equation}
    \label{eq:projective_estimate}
    \int_0^{2\pi}\max_{t\in[0,1]}\frac{1}{|re^{i\theta}+tc-c_k|^{\delta_k}}\frac{d\theta}{2\pi}
    \leq \frac{1}{r^{\delta_k}}\int_0^{2\pi}\frac{1}{|\cos\theta-|d_k|/r|^{\delta_k}}\frac{d\theta}{2\pi}.
\end{equation}
We have the following quadratic lower bounds
\begin{equation}
    \label{eq:quadratic_bound}
    |\cos\theta - \beta| 
    \geq 
    \left \{ 
    \begin{split}
        \frac{|1-\hat{\beta}|}{(\cos^{-1}(\hat{\beta}))^2}(\theta - \cos^{-1}(\hat{\beta}))^2, \quad \text{if }\theta\in[0,\cos^{-1}(\hat{\beta})]\\[6pt]
        \frac{|1+\hat{\beta}|}{(\pi-\cos^{-1}(\hat{\beta}))^2}(\theta - \cos^{-1}(\hat{\beta}))^2, \quad \text{if } \theta\in[\cos^{-1}(\hat{\beta}),\pi]
    \end{split}
    \right.
\end{equation}
where $\hat{\beta}=\min\{1,\max\{-1, \beta\}\}$ for fixed $\beta\in\mathbb{R}$, and we have the following linear lower bounds
\begin{equation}
    \label{eq:linear_bound}
    |\cos\theta - \beta| > \frac{\sqrt{1-\beta^2}}{2}|\theta-\cos^{-1}\beta|, \text{ for all } \theta\in[0,\pi]
\end{equation}
when $|\beta| < 1$. Note that $|\cos\theta - \beta|$ is $2\pi$-periodic and even about the $\theta=\pi$ axis, so these bounds can be repeated ad infinitum.

Substituting the quadratic bounds \eqref{eq:quadratic_bound} in \eqref{eq:projective_estimate}, and setting $\delta_k=\delta/2<1/2$ (due to the quadratic bound, the resulting integral does not converge unless $\delta_k<1/2$) we obtain the first assertion of this lemma. 

For the second assertion, we note that there exists large enough $R>0$ such that, for all $r>R$ we have $|c_k|/r<1+\sqrt{\epsilon}$ due to our assumptions on $c_k$, as we recall that $s=(\alpha+1)(r+|c|)/2$ and $\alpha\to 1$ as $r\to\infty$. In addition, by hypothesis
$$\left|\sin\left(\text{arg}\ \frac{c_k}{c}\right)\right|<1-\sqrt{\epsilon}$$
so we obtain
$$\frac{|d_k|}{r}=\frac{|c_k|}{r}|\sin(\text{arg}\ c_k - \text{arg}\ c)|< 1 - \epsilon.$$
Substituting the inequality above into \eqref{eq:linear_bound} with $\beta=|d_k|/r$ and substituting the resulting linear bounds into \eqref{eq:projective_estimate} and setting $\delta_k=\delta$ (for the linear bounds the resulting integral converges if $\delta_k<1$) yields the uniform bound of the second assertion.
\end{proof}
\end{lemma}

Using Lemma \ref{lemma:p1_z1_estimate} we obtain the following estimates for the $Z^1_\epsilon$ and $P^1_\epsilon$ parts of terms 3 and 4 

\begin{equation}
    \label{eq:Z1_estimate1}
    \begin{split}
    &\int_0^{2\pi}\max_{t\in[0,1]}\sum_{\substack{|a_k|<s\\a_k\in Z^1_\epsilon}}\log\left|\frac{s^2-\overline{a_k}re^{i\theta}}{s^2-\overline{a_k}(re^{i\theta}+tc)}\right|\frac{d\theta}{2\pi}\\[6pt]
    &\leq \frac{2|c|^{\delta/2}}{\delta}\frac{1}{1-\delta}\frac{1}{r^{\delta/2}}\left(n_{\angle\epsilon,c}(s, 1/g')-n_{\angle\epsilon,c}(s-\epsilon, 1/g')\right)\\[6pt]
    &\quad+\frac{|c|^\delta}{\delta}\left(\frac{2}{\sqrt{\epsilon(2-\epsilon)}}\right)^\delta\frac{1}{1-\delta}\frac{1}{r^{\delta}}\left(n(s, 1/g')-n(s-\epsilon, 1/g')\right)
    \end{split}
\end{equation}

and

\begin{equation}
    \label{eq:P1_estimate1}
    \begin{split}
    &\int_0^{2\pi}\max_{t\in[0,1]}\sum_{\substack{|b_k|<s\\b_k\in P^1_\epsilon}}\log\left|\frac{re^{i\theta}-b_k}{re^{i\theta}+tc-b_k}\right|\frac{d\theta}{2\pi}\\[6pt]
    &\leq \frac{2|c|^{\delta/2}}{\delta}\frac{1}{1-\delta}\frac{1}{r^{\delta/2}}\left(n_{\angle\epsilon,c}(r+|c|+\epsilon, g')-n_{\angle\epsilon,c}(r-|c|-\epsilon, g')\right)\\[6pt]
    &\quad+\frac{|c|^\delta}{\delta}\left(\frac{2}{\sqrt{\epsilon(2-\epsilon)}}\right)^\delta\frac{1}{1-\delta}\frac{1}{r^{\delta}}\left(n(r+|c|+\epsilon, g')-n(r-|c|-\epsilon, g')\right),
    \end{split}
\end{equation}
where we extend our exceptional set to include the bounded exceptional sets from our application of Lemma \ref{lemma:p1_z1_estimate}. By collecting all the estimates thus far, we arrive at the following intermediate result, which is useful in and of itself.

\begin{lemma}
Let $c\in\mathbb{C}\setminus\{0\}$, $a\in\mathbb{C}$ and let $f$ be a meromorphic function of hyper-order $<1$ that is not $c$-periodic. Then for $\epsilon\in(0,1)$, we have for all large enough $r>0$
    \begin{equation}
    \label{eq:no_exceptionalsets}
        \begin{split}
            &m\left(r,\frac{\Delta_c f-ac}{f'-a}\right)\leq K \left(\frac{\alpha+1}{\alpha-1}\right)\biggl[\frac{T(\alpha(r+|c|),f')}{r^{\delta}}\\&\quad+\left(R_{\epsilon,c}(\alpha(r+|c|),f') + R_{\epsilon,c}\left(\alpha(r+|c|),\frac{1}{f'-a}\right)\right)\frac{T(\alpha(r+|c|),f')}{r^{\delta/2}}\biggl],
        \end{split}
    \end{equation}
where $\alpha>1$ can either be a sufficiently small constant (how small depends only on $\epsilon$ and $c$) or $\alpha(r)>1$ and $\alpha(r)\to 1$ as $r\to\infty$, and where $K>0$ is some constant that depends on $\epsilon, \delta,\ f$ and $c$.
\end{lemma}

By further extending our exceptional sets, we can improve the results and get rid of the shift and scaling in the parameter of $T$ in the above result. Whereas in the estimates \eqref{eq:term2_estimate}, \eqref{eq:term5_estimate}, \eqref{eq:p2_estimate} and \eqref{eq:z2_estimate}, where we used conventional methods of estimating the unintegrated counting functions $n$ by integrated counting functions $N$, for the the estimates \eqref{eq:P1_estimate1} and \eqref{eq:Z1_estimate1} we gain extra leverage out of the assumption on the hyper-order of $f$. 

We may apply Lemma \ref{lemma:hyperorder_shift_estimate} to the unintegrated counting functions $n(\cdot,g')$ and $n(\cdot,1/g')$ (and similarly to $n_{\angle\epsilon,c}(\cdot,g')$ and $n_{\angle\epsilon,c}(\cdot,1/g')$), since they are increasing functions that can be continuously approximated up to arbitrary precision, and since their hyper-orders are less than or equal to the hyper-order of $g'$, which is less than $1$ by hypothesis:

\begin{align*}
    &\limsup_{r\to\infty}\frac{\log\log n(r,g')}{\log r}\leq\limsup_{r\to\infty}\frac{\log\frac{1}{\log (1+1/r^2)}N(r+1/r,g')}{\log r}\\
    &\leq\limsup_{r\to\infty}\frac{\log\left(2\log r+\log T(r+1/r,g')\right)}{\log r}\\
    &=\limsup_{r\to\infty}\frac{\log\log T(r+1/r,g') + \log\left(1 + \dfrac{2\log r}{\log T(r+1/r,g')}\right)}{\log r}\\
    &=\limsup_{r\to\infty}\frac{\log\log T(r+1/r,g')}{\log r}+\limsup_{r\to\infty}\frac{2\log r}{\log T(r+1/r,g')}\\
    &\leq\left(\lim_{r\to\infty}\frac{\log(r+1/r)}{\log r}\right)\left(\limsup_{r\to\infty}\frac{\log\log T(r+1/r,g')}{\log(r+1/r)}\right)\leq\xi.
\end{align*}
Thus, we apply Lemma \ref{lemma:hyperorder_shift_estimate} and obtain the following estimates
$$n(s,1/g')-n(s-\epsilon,1/g')\leq C_1\frac{n(s,1/g')}{r^{1-\xi-\lambda}},$$
and
$$n(r+|c|+\epsilon,g')-n(r-|c|-\epsilon,g')\leq C_2\frac{n(s,g')}{r^{1-\xi-\lambda}}$$
where $C_1,C_2>0$ and $\lambda\in(0, 1-\xi)$ are constants. Substituting these in \eqref{eq:Z1_estimate1} and \eqref{eq:P1_estimate1} yields the following estimates outside some exceptional set of finite logarithmic measure that depends on $\epsilon, \lambda$ and $g$:
\begin{equation}
    \label{eq:Z1_estimate2}
    \begin{split}
    &\int_0^{2\pi}\max_{t\in[0,1]}\sum_{\substack{|a_k|<s\\a_k\in Z^1_\epsilon}}\log\left|\frac{s^2-\overline{a_k}re^{i\theta}}{s^2-\overline{a_k}(re^{i\theta}+tc)}\right|\frac{d\theta}{2\pi}\\
    &\quad\leq C_3\frac{\alpha}{\alpha -1}\biggl(\frac{1}{r^\delta}+\frac{1}{r^{1+\delta/2-\xi-\lambda}}\frac{n_{\angle\epsilon,c}(s, 1/g')}{n(s, 1/g')}\biggl)N(\alpha(r+|c|),1/g')
    \end{split}
\end{equation}
and 
\begin{equation}
    \label{eq:P1_estimate2}
    \begin{split}
    &\int_0^{2\pi}\max_{t\in[0,1]}\sum_{\substack{|b_k|<s\\b_k\in P^1_\epsilon}}\log\left|\frac{re^{i\theta}-b_k}{re^{i\theta}+tc-b_k}\right|\frac{d\theta}{2\pi}\\
    &\quad\leq C_3\frac{\alpha}{\alpha -1}\biggl(\frac{1}{r^\delta}+\frac{1}{r^{1+\delta/2-\xi-\lambda}}\frac{n_{\angle\epsilon,c}(s, g')}{n(s, g')}\biggl)N(\alpha(r+|c|),g')
    \end{split}
\end{equation}
where $C_3$ is a sufficiently large constant. Including the resulting exceptional sets in our overall exceptional set, which is still of finite logarithmic measure, we collect the estimates \eqref{eq:term2_estimate}, \eqref{eq:term5_estimate}, \eqref{eq:p2_estimate}, \eqref{eq:z2_estimate}, \eqref{eq:Z1_estimate2} and \eqref{eq:P1_estimate2}, to obtain

\begin{equation}
    \label{eq:antepenultimate}
    \begin{split}
        m\left(r,\frac{\Delta_c g(re^{i\theta})}{g'(re^{i\theta})}\right)
        \leq C_4 \frac{\alpha+1}{\alpha-1}\biggl(&\frac{T(\alpha(r+|c|),g')}{r^\delta}
        +R_{\epsilon,c}(s,g')\frac{N(\alpha(r+|c|),g')}{r^{1+\delta/2-\xi-\lambda}}\\
        &+R_{\epsilon,c}\left(s,\frac{1}{g'}\right)\frac{N\left(\alpha(r+|c|),\frac{1}{g'}\right)}{r^{1+\delta/2-\xi-\lambda}}\biggl),
    \end{split}
\end{equation}
where $C_4$ is a sufficiently large constant. The rest of the proof borrows heavily from the proof of Theorem 5.1 in \cite{halburdkt:14TAMS}. With the choice
\begin{equation}
    \label{eq:alpha}
    \alpha=1+\frac{r+|c|}{(r+|c|)(\log T(r+|c|,g'))^{1+\lambda}},
\end{equation}
we seek to rid of the coefficient $\alpha$ in our arguments in \eqref{eq:antepenultimate}: we have the following estimate outside an exceptional set of finite logarithmic measure
\begin{equation}
    \label{eq:growth_lemma}
    T(\alpha(r+|c|),g')\leq CT(r+|c|,g')
\end{equation}
by \cite[Lemma 3.3.1]{cherryy:01}, and the same choice of $\alpha$ can be used to make similar estimates for the counting funtion terms $N$. Setting $\delta=1-\lambda$ in \eqref{eq:antepenultimate}, using \eqref{eq:growth_lemma} and estimating the terms involving $\alpha$ by \eqref{eq:alpha} and the fact that the hyper-order of $g'=f'-a$ is less than $1$, all together yields
\begin{equation}
    \label{eq:penultimate_estimate}
    \begin{split}
        &m\left(r,\frac{\Delta_c f(re^{i\theta})-ac}{f'(re^{i\theta})-a}\right)
        = O\left(\frac{T(r+|c|,f')}{r^{1-\xi-\lambda(1+\xi)}}\right)\\[6pt]
        &\quad+O\left(R_{\epsilon,c}(r+|c|,f')\frac{N\left(r+|c|,f'\right)}{r^{3/2-2\xi-\lambda(3/2+\xi)}}+R_{\epsilon,c}\left(r+|c|,\frac{1}{f'-a}\right)\frac{N\left(r+|c|,\frac{1}{f'-a}\right)}{r^{3/2-2\xi-\lambda(3/2+\xi)}}\right),
    \end{split}
\end{equation}
as $r\to\infty$. Further, by Lemma \ref{lemma:hyperorder_shift_estimate} we may also get rid of the shift in argument by $c$ in \eqref{eq:penultimate_estimate}: 
$$T(r+|c|,f')=T(r,f')+o\left(\frac{T(r,f')}{r^{1-\xi-\lambda}}\right),$$ 
and we similarly handle the shift in the remaining counting function terms in \eqref{eq:penultimate_estimate} (including the ones in $R_{\epsilon,c}$-terms). The assertion of the theorem is obtained by setting $\lambda=\epsilon/3$ and by including all the aforementioned exceptional sets of finite logarithmic measure in our final exceptional set. \qed

\section{Consequences}

\subsection{Deficiencies of meromorphic functions}
\label{defects}

\begin{proposition}
\label{prop:defect_inequality}
Let $c\in\mathbb{C}\setminus\{0\}$ and let $f$ be a meromorphic function of finite order $\rho<\infty$ that is not $c$-periodic. If the lower order $\mu(f')$ of $f'$ satisfies $\mu(f')>\rho(f)-1/2$, then
$$\delta(a,f')\leq \left(1+\limsup_{r\to\infty}\frac{N(r,f)}{T(r,f')}\right)\delta(ac,\Delta_c f).$$
In particular when $f$ is an entire function, we have
\begin{equation}
    \label{eq:deficiencies}
    \delta(a,f')\leq\delta\left(a,\frac{\Delta_c f}{c}\right).
\end{equation}
\end{proposition}
\begin{proof}
With the notation
\begin{equation}
    \label{eq:shorthand}
    A(a,r):=m\left(r,\frac{\Delta_c f-ac}{f'-a}\right)
\end{equation}
we have
$$m\left(r,\frac{1}{f'-a}\right)\leq m\left(r,\frac{1}{\Delta_c f-ac}\right)+A(a,r) \quad \text{and} \quad m(r,\Delta_c f)\leq m(r,f')+A(0,r).$$
Thus, for all large enough $r>0$
\begin{align*}
    &\frac{T(r,\Delta_c f)}{T(r,f')}\leq \frac{m(r,f')+N(r,f)+N(r,f(z+c))}{T(r,f')}+\frac{A(0,r)}{T(r,f')}\\[6pt]
    &\leq 1+\frac{N(r,f)}{T(r,f')}+\frac{O(r^{\rho(f')-1+\epsilon})+O(\log r)}{T(r,f')}+\frac{A(0,r)}{T(r,f')}.
\end{align*}
where we have used the following estimate \cite[Theorem 2.2]{chiangf:08}
$$N(r,f(z+c))=N(r,f)+O(r^{\lambda-1+\epsilon})+O(\log r)$$
where $\lambda$ is the exponent of convergence for the poles of $f$, and thus
\begin{align*}
    &\liminf_{r\to\infty}\frac{m\left(r,\frac{1}{f'-a}\right)}{T(r,f')}\leq\liminf_{r\to\infty}\left(\frac{m\left(r,\frac{1}{\Delta_c f-ac}\right)}{T(r,\Delta_c f)}\frac{T(r,\Delta_c f)}{T(r,f')}+\frac{A(a,r)}{T(r,f')}\right)\\&\leq\liminf_{r\to\infty}\biggl[\left(1+\frac{N(r,f)}{T(r,f')}+\frac{O(r^{\rho(f')-1+\epsilon})+O(\log r)+A(0,r)}{T(r,f')}\right)\\
    &\qquad\qquad\times\frac{m\left(r,\frac{1}{\Delta_c f-ac}\right)}{T(r,\Delta_c f)}+\frac{A(a,r)}{T(r,f')}\biggl].
\end{align*}
Now if $f$ satisfies $\mu(f')>\rho(f)-1/2$, then by setting $\delta=1-2\epsilon$ and by choosing a constant $\alpha>1$ close enough to $1$ in \eqref{eq:no_exceptionalsets}, we get for all large enough $r>0$
$$A(a,r)\leq\frac{3K_aT(\alpha(r+|c|),f')}{(\alpha-1)r^{1/2-\epsilon}}=O\left(\frac{T(r,f')}{r^\epsilon}\right),$$
where $\epsilon>0$ is a sufficiently small constant. A similar bound applies to $A(0,r)$. The first statement of the proposition follows from the properties of limit inferior and limit superior, and the second statement is an immediate consequence of the first.
\end{proof}

\begin{remark}
From \cite{kobayashi:76}, we have that if the zeros of an entire function $g$ lie on a sector whose size is less than $\pi/2$ and $1<\rho(g)<\infty$, then
\begin{equation}
    \label{eq:genus_vs_orders}
    \sigma(g)\leq\mu(g)\leq\rho(g)\leq\sigma(g)+1,
\end{equation}
where $\sigma(g)\in\mathbb{N}$ is the genus of $g$. Now if $\rho(g)$ is not an integer, then $\sigma(g)=\lfloor\rho(g)\rfloor$, so then also $\sigma(g')=\sigma(g)$, since $\rho(g')=\rho(g)$. If $\rho(g)$ is a positive integer, then either $\sigma(g)=\rho(g)$ or $\sigma(g)=\rho(g)-1$

If $f'$ has all its zeros on such a sector, $\rho(f)$ is not an integer, and the decimal part of $\rho(f)$ is less than $1/2$ or the decimal part of $\mu(f')$ is greater than $1/2$, then by the relation in \eqref{eq:genus_vs_orders} we see that $\mu(f')>\rho(f)-1/2$, which by the above investigations implies that $f$ satisfies \eqref{eq:deficiencies}.

Moreover, if the $a$-points of $f'$ all lie on a sector whose size is less than $\pi/2$, then substituting this fact into the $R_{\epsilon,c}$-terms in \eqref{eq:no_exceptionalsets} in the above arguments, we may replace the $1/r^{\delta/2}$ term by $1/r^\delta$ due to the lack of $a$-points outside the sector. Thus, we obtain \eqref{eq:deficiencies} for such $a\in\mathbb{C}$, whenever $ic$ and $-ic$ do not lie in the sector that contains the $a$-points of $f'$ and $\rho(f)$ is not an integer.
\end{remark}

\begin{remark}
For functions of hyper-order $\geq1$ with regular growth the function $e^{e^z}$ is a counter-example of \eqref{eq:deficiencies} for the value $a=0$ for $c\not\in 2\pi i\mathbb{Z}$.
\end{remark}

\subsection{c-separated pair indices of entire functions with finite order}

We recall here the definition of the $c$-separated index $\pi_c(a,f)$ for the value $a\in\mathbb{C}\cup\{\infty\}$ of a meromorphic function $f$ from \cite{halburdk:06AASFM}:

\begin{definition}[$c$-separated index]
\label{def:pair_counting_function}
Let $f$ be a meromorphic function that is not $c$-periodic. Then the $c$-separated index of $a\in\mathbb{C}\cup\{\infty\}$ is defined as
$$\pi_c(a,f) = \liminf_{r\to\infty}\frac{N_c(r,a,f)}{T(r,f)},$$
where the counting function 
$$N_c(r,a,f)=\int_0^r\frac{n_c(t,a,f)-n_c(0,a,f)}{t}dt+n_c(0,a,f)\log r$$
is defined such that $n_c(r,a,f)$ counts a point $|z_0|\leq r$ with $f(z_0)=a=f(z_0+c)$ (such points are called $c$-separated $a$-pairs) according to the number of equal terms at the beginning of the Taylor expansions of $f$ at $z_0$ and at $z_0+c$. For poles $n_c(r,f)$ counts the $c$-separated $0$-pairs of $1/f$.
\end{definition}

We obtain the following relation for entire functions with sufficiently regular growth.

\begin{proposition}
\label{prop:c_separated_index_inequality}
Let $c\in\mathbb{C}\setminus\{0\}$ and let $f$ be an entire function of finite order $\rho<\infty$ that is not $c$-periodic. If the lower order $\mu(f')$ of $f'$ satisfies $\mu(f')>\rho(f)-1/2$, then
$$\sum_{a\in\mathbb{C}}\pi_c(a,f)\leq 1-\delta(0,f').$$
\end{proposition}
\begin{proof}
Applying the continuous variable variant of Fatou's lemma with the counting measure (the proof of which is provided in the appendix below), we have
$$\sum_{a\in\mathbb{C}}\pi_c(a,f)
=\sum_{a\in\mathbb{C}}\liminf_{r\to\infty}\frac{N_c(r,a,f)}{T(r,f)}
\leq \liminf_{r\to\infty}\sum_{a\in\mathbb{C}}\frac{N_c(r,a,f)}{T(r,f)}.$$
Let $A$ be defined same as in \eqref{eq:shorthand} in Section \ref{defects}, so that we obtain
\begin{align*}
    &\sum_{a\in\mathbb{C}}\pi_c(a,f)+\Theta(0,f')
    \leq \liminf_{r\to\infty}\sum_{a\in\mathbb{C}}\frac{N_c(r,a,f)}{T(r,f)}-\limsup_{r\to\infty}\frac{\overline{N}(r,1/f')}{T(r,f')}+1\\[6pt]
    &\leq\liminf_{r\to\infty}\frac{N(r,1/\Delta_c f)}{T(r,f')}\frac{m(r,f')}{m(r,f)}+\liminf_{r\to\infty}\frac{-\overline{N}(r,1/f')}{T(r,f')}+1\\[6pt]
    &\leq\liminf_{r\to\infty}\frac{N(r,1/f')-\overline{N}(r,1/f')+A(0,r)}{T(r,f')}+1
    =\theta(0,f')+1,
\end{align*}
from which the statement immediately follows.
\end{proof}

\subsection{Analogue of 2nd main theorem for integrable functions}

For now assume that $f$ is a meromorphic function satisfying the requirements of Corollary \ref{corollary:ram_vs_pair_higher_order} that has a meromorphic primitive $F$. Then using the usual techniques, with Corollary \ref{corollary:ram_vs_pair_higher_order} acting similarly to the lemma on the logarithmic derivative, the following analogy of the second main theorem can be proved:

\begin{proposition}
If $f$ satisfies the hypothesis of Corollary \ref{corollary:ram_vs_pair_higher_order}, and has a primitive $F$, then
$$\sum_{i=1}^q m\left(r,\frac{1}{f-a_i'}\right)\leq m(r,f)+N(r,\Delta_c F)-N\left(r,\frac{1}{\Delta_c F}\right)+S(r,f)$$
outside an exceptional set of finite logarithmic measure, where $a_i$ are distinct $c$-periodic meromorphic functions such that $T(r,a_i)=S(r,f)$.
\end{proposition}

Note that in the proposition above, $a_i$ cannot be chosen such that $a_i'$ are constants, as such $a_i$ are not in the kernel of the operator $\Delta_c$, so in order to include constants in our results, we have to take a higher order difference.

\begin{proposition}
If $f$ satisfies the hypothesis of Corollary \ref{corollary:ram_vs_pair_higher_order}, and has a primitive $F$, then
$$m(r,f)+\sum_{i=1}^q m\left(r,\frac{1}{f-b_i'}\right)\leq 2T(r,f)+N(r,\Delta^2_c F)-2N(r,f)-N\left(r,\frac{1}{\Delta^2_c F}\right)+S(r,f)$$
outside an exceptional set of finite logarithmic measure, where $b_i=C_iz+a_i+\tilde{a}_iz$ are distinct meromorphic functions with $C_i\in\mathbb{C}$ and $a_i,\tilde{a}_i$ are $c$-periodic meromorphic functions with $T(r,a_i)=S(r,f)$ and $T(r,\tilde{a}_i)=S(r,f)$.
\end{proposition}

In the following we will consider $b_i=C_iz$, where $C_i\in\mathbb{C}$ are distinct constants, and seek to further refine the above result.

Recall the pair-counting function $N_c$ from Definition \ref{def:pair_counting_function}. Let $N_0(r,g)$ count a pole $z_0$ of a function $g$ if the Laurent expansions of $g(z_0)$ and $g(z_0+c)$ have equal principal parts, multiplicity
counted according to the number of equal terms in the beginning of the analytic
part of their expansions. Then from the above proposition,
and the fact that
$$N\left(r,\frac{1}{\Delta^2_c F}\right)\geq\sum_{i=1}^q N_c\left(r,\frac{1}{\Delta_c F-C_ic}\right)+N_0(r,\Delta_c F),$$
we have
\begin{align*}
    (q-1)T(r,f)&\leq N(r,f)+N(r,\Delta^2_c F)-2N(r,f)-N_0(r,\Delta_c F)\\[6pt]
    &\quad+\sum_{i=1}^q N\left(r,\frac{1}{f-C_i}\right)-N_c\left(r,\frac{1}{\Delta_c F-C_ic}\right)+S(r,f).
\end{align*}
Now in order to handle the poles and further simplify the notation, we make the following definitions
\begin{definition}
For non-$c$-periodic meromorphic $g$, define 
$$N_{\text{var}}(r,g)=\int_0^r \frac{n_{\text{var}}(t,g)-n_{\text{var}}(0,g)}{t}dt+n_{\text{var}}(0,g)\log r$$
where the contribution of a point $z_0\in \overline{D}(0,t)$ to $n_{\text{var}}(t,g)$ is
$$|\text{ord}_{-}(z_0) - \text{ord}_{-}(z_0+c)|-2m_{z_0},$$
with $m_{z_0}$ being the number of equal terms in the principal part of the Laurent expansions of $g$ at $z_0$ and $z_0+c$.
For $a\in\mathbb{C}$, define
$$\widehat{N}_c(r,a,g)=N\left(r,\frac{1}{g'-a}\right)-N_c\left(r,\frac{1}{\Delta_c g-ac}\right)$$
and similarly for poles
$$\widehat{N}_c(r,g)=\widehat{N}_c(r,\infty,g)=N(r,g')-N_c(r,\Delta_c g).$$
For $a\in\mathbb{C}\cup\{\infty\}$, define
$$\widehat{\Pi}_c(a,g')=1-\limsup_{r\to\infty} \frac{\widehat{N}_c(r,a,g)}{T(r,g')}.$$
\end{definition}
Continuing from where we left off, by using the above definitions, we get
\begin{align*}
    &N(r,\Delta^2_c F)-2N(r+|c|,F)-N_0(r,\Delta_c F)\\[6pt]
    &\leq \biggl(N(r,\Delta^2_c F)-2N(r+|c|,\Delta_c F)-N_0(r,\Delta_c F)\biggl)+S(r,\Delta_c F)\\
    &\quad+N(r,\Delta_c F)+\biggl(N(r,\Delta_c F)-2N(r+|c|,F)-N_0(r,F)\biggl)+N_0(r,F)\\[6pt]
    &\leq -N_c(r,\Delta_c F)+N(r,\Delta_c F)-N_c(r,F)+N_0(r,F)+S(r,f)\\[6pt]
    &\leq -N_c(r,\Delta_c F)+N_{\text{var}}(r,F)+S(r,f),
\end{align*}
where we have used the fact that for any meromorphic non-$c$-periodic $g$
$$N(r,\Delta_c g)-2N(r+|c|,g)-N_0(r,g)\leq-N_c(r,g)$$
proof of which is found in \cite[p. 471]{halburdk:06AASFM}, and similar calculations show that
$$N(r,\Delta_c F)-N_c(r,F)+N_0(r,F)=N_{\text{var}}(r,F).$$
Lemma \ref{lemma:hyperorder_shift_estimate} shows that we can interchange the terms $N(r,f)$ and $N(r+|c|,f)$ at the expense of an additional error term of growth class $S(r,f)$ and an exceptional set of finite logarithmic measure. Thus, we obtain
\begin{equation}
\label{eq:smt_truncated}
\begin{split}
    (q-1)T(r,f)&\leq\widehat{N}_c(r,F)+\sum_{i=1}^q \widehat{N}_c(r,C_i,F)\\
    &\quad+N_{\text{var}}(r,F)-2\overline{N}(r,F)+S(r,f)
\end{split}
\end{equation}
which implies the following relation. 

\begin{proposition}
\label{prop:defect_relation}
If $f$ satisfies the hypothesis of Corollary \ref{corollary:ram_vs_pair_higher_order}, and has a primitive $F$, then
\begin{equation}
\label{eq:defect_relation}
\sum_{a\in\mathbb{C}\cup\{\infty\}}\widehat{\Pi}_c(a,f)\leq 2\Theta(\infty,f) + \limsup_{r\to\infty}\frac{N_{var}(r,F)}{T(r,f)}.
\end{equation}

\begin{proof}
This follows from \eqref{eq:smt_truncated} by the same logic as the usual proof of the result
$$\sum_{a\in\mathbb{C}\cup\{\infty\}}\Theta(a,f)\leq 2$$
follows from the second main theorem (see \cite[Proof of Theorem 2.4]{hayman:64} for example).
\end{proof}

\end{proposition}

\begin{remark}
Note that $\Theta(\infty,f)\geq 1/2$. Further, it can be readily seen that $$2\Theta(\infty,f) + \limsup_{r\to\infty}\frac{N_{var}(r,F)}{T(r,f)}\geq0.$$
\end{remark}

\begin{example}
To see that the upper bound in Proposition \ref{prop:defect_relation} is attained, consider the meromorphic function $F(z)=e^z+z(W(z,c))^n$ where $W(\cdot,c)$ is the Weierstrass elliptic function with period $c\not\in 2\pi i\mathbb{Z}$ and $n\in\mathbb{N}$ and denote $f=F'$. 

Let us first find a lower bound for the left hand side of the 'defect relation' \eqref{eq:defect_relation} for $f$. Now
$$\Delta_c F(z)=\Delta_c e^z+c(W(z,c))^n \quad \text{and} \quad \Delta^2_c F(z)=\Delta^2_c e^z$$
Then for $a\in\mathbb{C}$ we have 
$$N_c(r,a,\Delta_c F)\leq N(r, 1/\Delta_c^2 F) = N(r, 1/\Delta_c^2 e^z)=0,$$
and therefore
$$0\leq \widehat{N}_c(r,a,F) = N(r,1/(f-a))-N_c(r,a,\Delta_c F)\leq T(r,f).$$
Thus
$$\sum_{a\in\mathbb{C}\cup\{\infty\}}\widehat{\Pi}_c(a,f)\geq\widehat{\Pi}_c(\infty,f).$$
Since $\Delta_c F$ has only $c$-separated poles with identical Laurent series principal parts, we get
$$N_c(r,\Delta_c F)=2N(r,\Delta_c F)=2N(r,F).$$
In total we have the following estimate for the left hand side of \eqref{eq:defect_relation}:
\begin{align*}
   & \sum_{a\in\mathbb{C}\cup\{\infty\}}\widehat{\Pi}_c(a,f)
    \geq\widehat{\Pi}_c(\infty,f)\\
    &=1-\limsup_{r\to\infty} \frac{\widehat{N}_c(r,\infty,F)}{T(r,f)}
    =1-\limsup_{r\to\infty} \frac{N(r,f)-N_c(r,\Delta_c F)}{T(r,f)}\\
    &=1-\limsup_{r\to\infty} \frac{N(r,f)-2N(r,F)}{T(r,f)}
    =1+\lim_{r\to\infty} \frac{N(r,F)-\overline{N}(r,F)}{N(r,F)+\overline{N}(r,F)}\\
    &=1+\frac{2n-1}{2n+1}
    =\frac{4n}{2n+1},
\end{align*}
because the poles are not deficient. For the right hand side of \eqref{eq:defect_relation}, we have
$$2\Theta(\infty,f) + \limsup_{r\to\infty}\frac{N_{var}(r,f)}{T(r,f)}=2\Theta(\infty,f)=2\left(1-\frac{1}{2n+1}\right)=\frac{4n}{2n+1}$$
since $N_{\text{var}}(r,F)=0$ as all poles of $F$ are $c$-separated, of equal order and do not have equal principal parts. Thus, in summary, we have
$$\frac{4n}{2n+1}\leq \sum_{a\in\mathbb{C}\cup\{\infty\}}\widehat{\Pi}_c(a,f) \leq \frac{4n}{2n+1}$$
which implies that the bound in \eqref{eq:defect_relation} is attained:
$$\sum_{a\in\mathbb{C}\cup\{\infty\}}\widehat{\Pi}_c(a,f) = \frac{4n}{2n+1}=2\Theta(\infty,f) + \limsup_{r\to\infty}\frac{N_{var}(r,f)}{T(r,f)}.$$
Moreover, as $n\to\infty$, upper bounds arbitrarily close to 2 are attained by these functions in \eqref{eq:defect_relation}.
\end{example}

\subsection{Delay differential equations}

In this section we prove two results concerning meromorphic solutions to delay differential equations; the first one analogous to the classic Clunie Lemma~\cite[Lemma~2]{clunie:62} and the second one analogous to the classic result due to A. Z. Mohon’ko and V. Z. Mohon’ko \cite{mohonko:74}.

Recall the conventional formulation of the Clunie Lemma:

\begin{citelemma}[Clunie Lemma]
Let $n\in\mathbb{N}$ and let $f$ be a transcendental meromorphic solution to
$$f^nP(z,f)=Q(z,f)$$
where $P$ and $Q$ are polynomials in $f, f', f'', \cdots$ with small meromorphic coefficients. If the degree of $Q$ in $f$ and its derivatives is at most $n$, then 
$$m(r,P)=S(r,f)$$ 
where the exception set implied in $S(r,f)$ is one of finite linear measure.
\end{citelemma}

The following analogue can be proven with the help of Corollary \ref{corollary:ram_vs_pair_higher_order}.

\begin{proposition}
Let $n\in\mathbb{N}$ and let $f$ be a meromorphic solution to 
$$(f')^nP(z,f)=Q(z,f)$$
where $P$ and $Q$ are polynomials in $\Delta_{c_1}f,\cdots,\Delta_{c_m}f$ with small meromorphic coefficients. If $f$ satisfies the hypothesis of Corollary \ref{corollary:ram_vs_pair_higher_order} for $c_1,\cdots,c_m$, and if the degree of $Q$ in the differences of $f$ is at most $n$, then
$$m(r, P)=S(r,f)$$
outside a set of finite logarithmic measure.

\begin{proof}
This requires just a slight modification to the original proof of Clunie's lemma: Split the domain of the integration in $m(r, P)$ into two parts $E$ and $[0,2\pi]\setminus E$, where 
$$E=\left\{ \theta\in[0,2\pi]: |f'(re^{i\theta})|<1 \right\}.$$
Write $P=\sum_{\lambda}P_\lambda=\sum_{\lambda}a_{\lambda}(\Delta_{c_1}f)^{l_1}\cdots(\Delta_{c_m}f)^{l_m}$ where the sum is over a set of multi-indices $\lambda=(l_1,\cdots,l_m)$. Then for $z\in E$, we have
$$|P_\lambda(z,f)|=|a_\lambda||\Delta_{c_1}f|^{l_1}\cdots|\Delta_{c_m}f|^{l_m}\leq |a_\lambda|\left|\frac{\Delta_{c_1}f}{f'}\right|^{l_1}\cdots\left|\frac{\Delta_{c_m}f}{f'}\right|^{l_m},$$
so that Corollary \ref{corollary:ram_vs_pair_higher_order} implies that the $E$-part of the integral in $m(r, P)$ is $S(r,f).$
For $z\in[0,2\pi]\setminus E,$ write $Q=\sum_{\lambda}Q_\lambda=\sum_{\lambda}b_{\lambda}(\Delta_{c_1}f)^{l_1}\cdots(\Delta_{c_m}f)^{l_m}.$ Now
\begin{align*} &|P(z,f)|\leq|1/f'|^n\sum_{\lambda}|b_\lambda||\Delta_{c_1}f|^{l_1}\cdots|\Delta_{c_m}f|^{l_m}\\
&\leq\sum_{\lambda}|b_\lambda|\left|\frac{\Delta_{c_1}f}{f'}\right|^{l_1}\cdots\left|\frac{\Delta_{c_m}f}{f'}\right|^{l_m},
\end{align*}
since by hypothesis for any $\lambda=(l_1,\cdots,l_m)$ in $Q$ we have $l_1+\cdots+l_m\leq n$. Again by Corollary \ref{corollary:ram_vs_pair_higher_order} this implies that the $([0,2\pi]\setminus E)$-part of the integral in $m(r, P)$ is $S(r,f),$ which concludes the proof.
\end{proof}
\end{proposition}

The classic result due to A. Z. Mohon’ko and V. Z. Mohon’ko in \cite{mohonko:74} conserns solutions to differential equations of the form
$P(z,f,f',\cdots,f^{(n)})=0$, where $P$ is polynomial in $f, f', \cdots, f^{(n)}$ with small meromorphic coefficients in $z$. The result states that if a constant $a\in\mathbb{C}$ does not solve $P(z,a,0,\cdots,0)=0$, then $m(r,1/(f-a))=S(r,f)$. Further, using the difference analogue of the lemma on the logarithmic derivative, one can prove a difference analogue of the Mohon'ko result. We have the following result, which is of a similar character.

\begin{proposition}
\label{prop:mohonko-like}
Let $f$ be a meromorphic solution  to 
\begin{equation}
\label{eq:mohonko_eq}
\sum_{j=1}^m a_j\Delta_{c_j} f=f'R(z,f)=f'\frac{P(z,f)}{Q(z,f)},
\end{equation}
where the meromorphic coefficient functions $a_j$ satisfy $m(r,a_j)=S(r,f)$, and where $P$ and $Q$ are polynomial in $f$ with rational coefficients, such that
$$Q(z,f)=(f-b_1)\cdots(f-b_k),$$
where $b_j(z)$ are distinct rational functions. If $f$ satisfies the hypothesis of Corollary~\ref{corollary:ram_vs_pair_higher_order} for each $c_j$ and $n=1$ and $\deg_f(P)\leq\deg_f(Q)$, then
$$\sum_{j=1}^k N\left(r,\frac{1}{f-b_j}\right)=\deg_f(R)T(r,f)+S(r,f).$$
\begin{proof}
We have
\begin{align*}
    m(r,P/Q)&=T(r,P/Q)-N(r,P/Q)\\
    &=\deg_f(R)T(r,f)-\sum_{j=1}^k N\left(r,\frac{1}{f-b_j}\right)\\
    &=\deg_f(R)T(r,f)-kT(r,f)+\sum_{j=1}^k m\left(r,\frac{1}{f-b_j}\right)\\
    &=\sum_{j=1}^k m\left(r,\frac{1}{f-b_j}\right)=\sum_{j=1}^m m(r,\Delta_{c_j} f/f')+m(r,a_j)=S(r,f),
\end{align*}
where we have used the fact that $k=\deg_f(R)$ which follows from the assumption that $\deg_f(P)\leq\deg_f(Q)$.
\end{proof}
\end{proposition}

\begin{remark}
The polynomial formulation 
\begin{equation}
\label{eq:mohonkoexample1}
\tilde{P}(z,f,f',\Delta_{c_1}f,\cdots,\Delta_{c_m}f)=\left(\sum_{j=1}^m a_j\Delta_{c_j}f\right)Q(z,f)-f'P(z,f)=0
\end{equation}
of \eqref{eq:mohonko_eq} is solved by any constant, so the usual difference analogue of Mohon'kos' result cannot be applied directly. However Proposition \ref{prop:mohonko-like} does imply that any solution of \eqref{eq:mohonkoexample1} that satisfies the hypothesis of the proposition must satisfy
$$\sum_{j=1}^km\left(r,\frac{1}{f-b_j}\right)=S(r,f).$$
\end{remark}
\begin{example}
Consider the equation introduced by Quispel, Capel and Sahadevan in \cite{quispelcs:92}
\begin{equation}
    \label{eq:painleve_original}
    w(z)[w(z+1)-w(z-1))] + aw'(z) = bw(z),
\end{equation}
where $a$ and $b$ are constants. Quispel et al. obtained this equation as a symmetry reduction of the Kac-van Moerbeke equation. They showed that it has a Lax pair, possesses a continuum limit to the first Painlevé equation, and, in the special case where $b=0$, the equation exhibits particular soliton and rational solutions. The special case $b=0$ can be rewritten in the form
\begin{equation}
\label{eq:painleve_recast}
    \Delta_1 w(z)-\Delta_{-1} w(z) = \frac{-aw'(z)}{w(z)}.
\end{equation}
If $w$ is a solution of \eqref{eq:painleve_recast} that satisfies the hypothesis of Proposition \ref{prop:mohonko-like}, then the proposition implies that
$$N(r,1/w)=T(r,1/w)+S(r,w).$$
\end{example}

\section*{Appendix}
\begin{citelemma}[Continuous variable variant of Fatou's lemma with the counting measure]
Let $f_k$ be a collection of bounded non-negative continuous functions. Then
$$\sum_{k=1}^\infty \liminf_{r\to\infty} f_k(r)\leq\liminf_{r\to\infty}\sum_{k=1}^\infty f_k(r).$$
\end{citelemma}
\begin{proof}[Proof of variant of Fatou's lemma:]
Let $(r_j)_{j\in\mathbb{N}}$ be a sequence such that $r_j\to\infty$ as $j\to\infty$ and 
$$\lim_{j\to\infty}\sum_{k=1}^\infty f_k(r_j)=\liminf_{r\to\infty}\sum_{k=1}^\infty f_k(r),$$
which exists by the definition of limit inferior. Now

\begin{align*}
    &\sum_{k=1}^\infty \liminf_{r\to\infty} f_k(r)
    \leq\sum_{k=1}^\infty \liminf_{j\to\infty} f_k(r_j)
    \leq\liminf_{j\to\infty}\sum_{k=1}^\infty f_k(r_j)\\
    &=\lim_{j\to\infty}\sum_{k=1}^\infty f_k(r_j)
    =\lim_{r\to\infty}\sum_{k=1}^\infty f_k(r)
\end{align*}
where the second inequality follows by the Fatou lemma.
\end{proof}

\bibliographystyle{plain} 
\bibliography{database} 

\end{document}